\documentclass{amsart}
\usepackage{amsfonts, amsmath, amssymb, amsthm}
\usepackage{graphicx}
\usepackage{todonotes}
\newtheorem{theorem}{Theorem}[section]
\newtheorem{remark}[theorem]{ Remark}

\newtheorem{corollary}[theorem]{Corollary}

\newtheorem{definition}[theorem]{Definition}

\newcommand{\V}{\Vert}
\newcommand{\RR} {\mathbb R}

\newcommand{\pa} {\partial}
\newcommand{\Cal} {\mathcal}

\newcommand{\beq} {\begin{equation}}
\newcommand{\eeq} {\end{equation}}

\newcommand{\Vol}{\operatorname{Vol}}

\begin{document}
\title[Nodal geometry, heat diffusion and Brownian motion]{ Nodal geometry, heat diffusion and Brownian motion}
\author{Bogdan Georgiev and Mayukh Mukherjee}

\address{Max Planck Institute for Mathematics\\ Vivatsgasse 7\\ 53111 Bonn,
\\ Germany}

\email{bogeor@mpim-bonn.mpg.de}
\email{mukherjee@mpim-bonn.mpg.de}
\begin{abstract}

We use tools from $n$-dimensional Brownian motion in conjunction with the Feynman-Kac formulation of heat diffusion to study nodal geometry on a compact Riemannian manifold $M$. On one hand we extend a theorem of Lieb (see \cite{L}) and prove that 
any nodal domain $\Omega_\lambda$ almost fully contains a ball of radius $\sim \frac{1}{\sqrt{\lambda}}$, 
which is made precise by Theorem \ref{th:Large-Inscribed-Ball} below. This also gives a slight refinement of a result by Mangoubi, concerning the inradius of nodal domains (\cite{Man2}). On the other hand, we also prove that no nodal domain can be contained in a reasonably narrow tubular neighbourhood of unions of finitely many submanifolds inside $M$ (this is Theorem \ref{stein}). 
\end{abstract}
\maketitle
\section{Introduction}
We consider a compact $n$-dimensional smooth Riemannian manifold $M$, and the Laplacian (or the Laplace-Beltrami operator) $-\Delta$ on $M$\footnote{We use the analyst's sign convention, namely, $-\Delta$ is positive semidefinite.}. For an eigenvalue $\lambda$ of $-\Delta$ and a corresponding eigenfunction $\varphi_\lambda$, 
recall that a nodal domain $ \Omega_\lambda $ is a connected component of the complement of the nodal set $N_{\varphi_\lambda} := \{ x \in M : \varphi_\lambda (x) = 0\}$. In this paper, we are interested in the asymptotic geometry of a nodal domain $\Omega_\lambda$ as $ \lambda \rightarrow \infty $.

In this note we address the following two questions.

First, we start by discussing the problem of whether a nodal domain can be squeezed in a tubular neighbourhood around a certain subset $ \Sigma \subseteq M $. A result of Steinerberger (see Theorem 2 of ~\cite{St}) states that for some constant $ r_0>0 $ a nodal domain $ \Omega_\lambda $ cannot be contained in a $ 
\frac{r_0}{\sqrt{\lambda}} $-tubular neighbourhood of hypersurface $ \Sigma $, provided that $ \Sigma $ is sufficiently flat in the following sense: $ \Sigma $ must admit a unique metric projection in a wavelength (i.e. $\sim \frac{1}{\sqrt{\lambda}} $) tubular neighbourhood.
The proof involves the study of a heat process associated to the nodal domain, where one also uses estimates for Brownian motion and the Feynman-Kac formula.

We relax the conditions imposed on $ \Sigma $. Our first result is a direct extension of Theorem 2 of \cite{St}. Before stating the result, we begin with the following definition: 
\begin{definition}[Admissible Collections]\label{adm}
	For each fixed eigenvalue $ \lambda $, we consider a natural number $ m_\lambda \in \mathbb{N} $ and a collection $\Sigma_\lambda := \cup^{m_\lambda}_{i = 1} \Sigma^i_\lambda$, where $ \Sigma^i_\lambda $ is an embedded smooth submanifold (without boundary) of dimension $ k $, ($ 1\leq k \leq n-1 $).
	
	We call $\Sigma_\lambda$ admissible up to a distance $r$ if the following property is satisfied: for any $ x \in M $ with $\text{dist} (x, \Sigma_\lambda) \leq r$ there exists a unique index $ 1 \leq i_x(\lambda) \leq m_\lambda $ and a unique point $y \in \Sigma_\lambda^{i_x(\lambda)}$ realizing $\text{dist}(x, \Sigma_\lambda)$ - that is, $\text{dist} (x, y) = \text{dist}(x, \Sigma_\lambda)$.
\end{definition}

We note that if $\Sigma_\lambda$ consists of one submanifold which is admissible up to distance $r$, then Definition \ref{adm} means that $r$ is smaller than the normal injectivity radius of $\Sigma_\lambda$. Moreover, if $ \Sigma_\lambda $ consists of more submanifolds, then these submanifolds must be disjoint and the distance between every two of them must be greater than $r$.

Let us also remark that, Theorem 2 of ~\cite{St} holds true when the hypersurface $\Sigma$ is allowed to vary with respect to $\lambda$ in a controlled way, which is made precise by Definition \ref{adm}. With that clarification in place, Theorem \ref{babycase} is an extension of Theorem 2 of ~\cite{St}.

\begin{theorem}\label{babycase}
	There is a constant $r_0$ depending only on $(M, g)$ such that if a submanifold $\Sigma_\lambda \subset M$ is admissible up to distance $\frac{1}{\sqrt{\lambda}}$, then no nodal domain $\Omega_\lambda$ can be contained in a $\frac{r_0}{\sqrt{\lambda}}$-tubular neighbourhood of $\Sigma_\lambda$.
\end{theorem}


Further, it turns out that we can select $ \Sigma_\lambda $ to be a union of submanifolds of varying dimensions, having relaxed admissibility conditions.

Elaborating on this, we observe that getting entirely rid of the admissibility condition, as in Definition \ref{adm} allows situations where $\Sigma^i_\lambda$ is dense in $M$, for example, $M = \mathbb{T}^2$ and $\Sigma^1_\lambda$ being a generic geodesic. By assuming $\Sigma^i_\lambda$ is compact, we avoid such situations. Also, since we are considering unions of surfaces, the restriction of ``unique projection'' of nearby points, as in Definition \ref{adm}, makes no sense any more, and one can see that the approach of the proof of Theorem \ref{babycase} does not work.

 First, for ease of presentation, we adopt the following notation.
\begin{definition}
	Given a compact subset $K$ of $M$, let $\psi_K(t, x)$ denote the probability that a particle undergoing a Brownian motion starting at the point $x$ will reach $K$ within time $t$.
\end{definition}

We now introduce the following relaxed notion of admissibility.
\begin{definition}[$\alpha $-admissible Collections] \label{def:alpha-adm-col} \label{def:Alpha-Admissibility}
	Let $ 0 < \alpha < 1 $ be a constant. For each fixed eigenvalue $ \lambda $, we consider a natural number $ m_\lambda \in \mathbb{N} $ and a collection $\Sigma_\lambda := \cup^{m_\lambda}_{i = 1} \Sigma^i_\lambda$, where $ \Sigma^i_\lambda $ is a compact embedded smooth submanifold (without boundary) of dimension $ k_i $, ($ 1\leq k_i \leq n-1 $). Denote the respective tubular neighbourhoods by $N_\varepsilon(\Sigma^i_\lambda) := \{x \in M : \text{dist }(x, \Sigma^i_\lambda) < \varepsilon\}$, and let $N_\varepsilon(\Sigma_\lambda) = \cup^{m_\lambda}_{i = 1}N_\varepsilon(\Sigma^i_\lambda)$.
	
	We say that the collection $ \Sigma_\lambda $ is $ \alpha $-admissible, if for each sufficiently small $ \varepsilon > 0 $ and each $ x \in N_\varepsilon(\Sigma_\lambda) $ we have
	\begin{equation} \label{eq:alpha-adm}
		\psi_{\partial B(x, 2 \epsilon) \setminus N_\varepsilon(\Sigma_\lambda)} (4 \varepsilon^2, x) \geq \alpha \psi_{\partial B(x, 2 \epsilon)} (4 \varepsilon^2, x).
	\end{equation}
	
\end{definition}

Intuitively, using the above implicit formulation via Brownian motion hitting probabilities, we wish to ensure that $ N_\varepsilon(\Sigma_\lambda) $ does not occupy too large a proportion of each $ B(x, 2\epsilon) $ for $ x \in N_\varepsilon(\Sigma_\lambda) $ (cf. diagram on page \pageref{diag2} below).

In other words, we allow the family $ \Sigma_\lambda $ to intersect, but the intersections should not be ``too dense''. To illustrate the idea, let us for simplicity assume that $ M = \mathbb{R}^n $ and let us suppose that each member $ \Sigma^i_\lambda $ of the collection $ \Sigma_\lambda $ is a line passing through the origin. If the collection of these lines gets sufficiently close together or in other words ``dense'', then no matter how small $ \varepsilon > 0 $ we take, the tubular neighbourhood $ N_\varepsilon(\Sigma_\lambda) $ will contain the ball $ B(0, 2 \varepsilon) $. In particular, the left hand side of $ (\ref{eq:alpha-adm}) $ is vanishing and so, there is no $ \alpha > 0 $ for which the collection $ \Sigma_\lambda $ is $ \alpha $-admissible. Clearly, in the above example, replacing the lines $ \Sigma^i_\lambda $ by linear subspaces of varying dimensions will deliver a similar example of a collection, which is not $ \alpha $-admissible.

Having this intuition in mind, we have the following result.

\begin{theorem}\label{stein}
	Given an $ \alpha $-admissible collection $\Sigma_\lambda$, there exists a constant $C$, independent of $\lambda$, such that $N_{\frac{C}{\sqrt{\lambda}}}(\Sigma_\lambda)$ cannot fully contain a nodal domain $\Omega_\lambda$.
\end{theorem}

Theorem \ref{stein} gives a strong indication as to the ``thickness'' or general shape of a nodal domain in many situations of practical interest. For example, in dimension $2$, numerics show nodal domains to look like a tubular neighbourhood of a tree. We also note that our proof of Theorem \ref{stein} reveals a bit more information, but for aesthetic reasons, we prefer to state the theorem this way.
Heuristically, the proof reveals that the nodal domain $\Omega_\lambda$ is thicker at the points where the eigenfunction $\varphi_\lambda$ attains its maximum, or at points where $\varphi_\lambda (x) \geq \beta\text{max}_{y \in \Omega_\lambda}|\varphi_\lambda (y)|$, for a fixed constant $\beta > 0$. 

Second, we study the problem of how large a ball one may inscribe in a nodal domain $ \Omega_\lambda $ at a point where the eigenfunction achieves extremal values on $ \Omega_\lambda $. We show

\begin{theorem} \label{th:Large-Inscribed-Ball}
Let $\dim M \geq 3, \epsilon_0 > 0 $ be fixed and $ x_0 \in \Omega_\lambda $ be such that $ |\varphi_\lambda(x_0)| = max_{\Omega_\lambda}|\varphi_\lambda| $. There exists $ r_0 = r_0 (\epsilon_0) $, such that
	\begin{equation}\label{Vol}
		\frac{\Vol\left( B(x_0, r_0 \lambda^{-1/2}) \cap \Omega_\lambda \right)}{\Vol\left( B(x_0, r_0 \lambda^{-1/2})\right)} \geq 1 - \epsilon_0.
	\end{equation}
\end{theorem}


A celebrated theorem of Lieb (see \cite{L}) considers the case of a domain $ \Omega \subset \mathbb{R}^n $ and states that there exists a point $ x_0 \in \Omega $, where a ball of radius $ \frac{C}{\sqrt{\lambda_1(\Omega)}} $ can almost be inscribed (in the sense of our Theorem \ref{th:Large-Inscribed-Ball}). A further generalization was obtained in the paper ~\cite{MS} (see, in particular, Theorem 1.1 and Subsection 5.1 of ~\cite{MS}). However, the point $ x_0 $ was not specified. Physically, one expects that $ x_0 $ is close to the point where the first Dirichlet eigenfunction of $ \Omega $ attains extremal values. This is in fact the essential statement of Theorem \ref{th:Large-Inscribed-Ball} above. 
Also, in this context, it is illuminating to compare the main Theorem from ~\cite{CD}.

We take the space to reiterate that the proof of Theorem \ref{th:Large-Inscribed-Ball} uses estimates from \cite{GS} (see (\ref{Maz})), and a certain isocapacitary estimate (see (\ref{II})) that work only in dimensions $n \geq 3$. As far as dimension $n = 2$ is concerned, it is known due to Mangoubi (Theorem 1.2 of \cite{Man1}, see also \cite{H}) that any nodal domain has wavelength inradius; see further discussion on this at the beginning of Section \ref{LBMP}. 

As a corollary of Theorem \ref{th:Large-Inscribed-Ball}, we derive the following:
\begin{corollary} \label{cor:Large-Ball}
	Let $M$ be a closed manifold of dimension $n \geq 3$, and $\Omega_\lambda \subseteq M $ be a nodal domain upon which the corresponding eigenfunction $ \varphi_\lambda $ is positive. Let $ x_0 $ be a point of maximum of $ \varphi_\lambda $ on $ \Omega_\lambda $. Then there exists a ball $ B\left(x_0, \frac{C}{\lambda^{\alpha(n)}}\right) \subseteq \Omega_\lambda $ with $ \alpha (n) = \frac{1}{4}(n-1)+\frac{1}{2n} $ and a constant $ C = C(M, g) $.
\end{corollary}
This recovers Theorem 1.5 of ~\cite{Man2}, with the additional information that the ball of radius $\frac{C}{\lambda^{\alpha(n)}}$ is centered around the max point of the eigenfunction $\varphi_\lambda$ (for more discussion on this, see Section \ref{LBMP}). We also point out that using Theorem \ref{th:Large-Inscribed-Ball}, the first author has established in \cite{G} using results from \cite{JM}, the following inner radius bounds for real analytic manifolds:
\begin{theorem}[\cite{G}]
	Let $(M, g)$ be a real-analytic closed manifold of dimension at least $3$. Let $\varphi_\lambda$ be a Laplacian eigenfunction and $\Omega_\lambda$ be a nodal domain of $\varphi_\lambda$. Then, there exist constants $c_1, c_2$ depending only on $(M, g)$ such that 
	$$ \frac{c_1}{\lambda} \leq \text{inrad} (\Omega_\lambda) \leq \frac{c_2}{\sqrt{\lambda}}.$$
	Moreover, if $\varphi_\lambda$ is positive (resp. negative) on $\Omega_\lambda$, then a ball of this radius can be inscribed within a wavelength distance to a point where $\varphi_\lambda$ achieves its maximum (resp. minimum) on $\Omega_\lambda$.
	
\end{theorem} 

For another improvement of inner radius estimates in the smooth setting under certain conditional bounds on $\| \varphi_\lambda\|_{L^\infty(\Omega_\lambda)}$, see Theorem 1.7 of \cite{GM}.

A few assorted remarks: as advertised, in Section \ref{Perf} we address the problem of inscribing a nodal domain $ \Omega_\lambda $ in a tubular neighbourhood around $ \Sigma $. 
In this context, an interesting subcase one might also consider is $ \Sigma $ having conical singularities: at its singular points $ \Sigma $ looks locally like $ \mathbb{R}^{n-1-k} \times \partial C^k $ for some $ k=1,\dots, n-1 $, where $ \partial C^k $ denotes the boundary of a generalized cone, i.e. the cone generated by some open set $ D \subseteq \mathbb{S}^{n-1}$.

In this situation a useful tool is an explicit heat kernel formula for generalized cones $ C \subseteq \mathbb{R}^n $. One denotes the associated Dirichlet eigenfunctions and eigenvalues of the generating set $ D $ by $m_j, l_j $ respectively. Using polar coordinates $ x = \rho \theta, y = r \eta $, one has that the heat kernel of $ P_C(t,x,y) $ of the generalized cone $ C $ is given by
\begin{equation} \label{eq:Heat-Kernel-Cones}
P_C(t,x,y)= \frac{e^{-\frac{\rho^2+r^2}{2t}}}{t(\rho r)^{\frac{n}{2}+1}} \sum_{j=1}^{\infty} I_{\sqrt{l_j+(\frac{n}{2}-1)^2}} (\frac{\rho r}{t}) m_j(\theta) m_j(\eta),
\end{equation}
where $ I_\nu(z) $ denotes the modified Bessel function of order $ \nu $.
For more on the formula (\ref{eq:Heat-Kernel-Cones}) we refer to \cite{BS}. An even more general formula can be found in ~\cite{C}.

The expression for $ P_C(t,x,y) $ provides means for estimating $ p_t(x) $ from below as in Section \ref{Perf}. However, some features of the conical singularity (i.e. the eigenvalues and eigenfunctions $ l_j, m_j $ of the generating set $ D $) enter explicitly in the estimate. Such considerations appear promising in discussing theorems of the following type, for example, and their higher dimensional analogues (see also ~\cite{St}):
\begin{theorem}[Bers, Cheng] Let $n = 2$. If $-\Delta u = \lambda u$, then any nodal set satisfies an interior cone condition with opening angle $\alpha \gtrsim \lambda^{-1/2}$.
\end{theorem}

\subsection{Basic heuristics}
We outline the main idea behind Theorems \ref{babycase}, \ref{stein} and \ref{th:Large-Inscribed-Ball}.

First, one considers a point $ x_0 \in \Omega_\lambda$ where the eigenfunction achieves a maximum on the nodal domain (w.l.o.g. we assume that the eigenfunction is positive on $ \Omega_\lambda $). One then considers the quantity $ p(t,x_0) $ - i.e. the probability that a Brownian motion started at $ x_0 $ escapes the nodal domain within time $ t $.

The main strategy is to obtain two-sided bounds for $ p(t, x_0) $.

On one hand, we have the Feynman-Kac formula (see Subsection \ref{F-K}) which provides a straightforward upper bound only in terms of $ t $ (see Equation (\ref{stein1}) below).

On the other hand, depending on the context of the theorems above, we provide a lower bound for $ p(t,x_0) $ in terms of some geometric data. To this end, we take advantage of various tools some of which are: formulas for hitting probabilities of spheres and the parabolic scaling between the space and time variables; comparability of Brownian motions on manifolds with similar geometry (see Subsection \ref{Eu-comp}); bounds for hitting probabilities in terms of $ 2 $-capacity (cf. \cite{GS}), etc.

\subsection{Outline of the paper}
In Section \ref{Int}, we recall tools from $n$-dimensional Brownian motion and the Feynman-Kac formulation of heat diffusion, and discuss the parabolic scaling technique we referred to above. We include some background material on stochastic analysis on Riemannian manifolds, some of which (to our knowledge) is not widely known, but is important to our investigation. We also believe such results to be of independent interest to the community. Of particular mention is Theorem \ref{compare}, which roughly says that if the metric is perturbed slightly, hitting probabilities of compact sets by Brownian particles are also perturbed slightly. This allows us to apply Brownian motion formulae from $\RR^n$ to compact manifolds, on small distance and time scales. 

In Section \ref{Perf}, we begin by proving Theorem \ref{babycase}. 
As mentioned before, we then take the generalization 
 one step further, by considering intersecting surfaces of different dimensions. Our main result in this direction is Theorem \ref{stein}, which gives a quantitative lower bound on how ``thin'' or ``narrow'' a nodal domain can be. 

In Section \ref{LBMP}, we take up the investigation of inradius estimates of $\Omega_\lambda$. As mentioned before, our main result in this direction is Theorem \ref{th:Large-Inscribed-Ball}. We also establish Corollary \ref{cor:Large-Ball}.




\subsection{Acknowledgements} It is a pleasure to thank Stefan Steinerberger for his detailed comments on a draft version of this paper, as well as Dan Mangoubi for his comments and remarks. The authors further thank Yuval Peres and Itai Benjamini for advice regarding Martin capacity, as well as Steve Zelditch for discussions on the Feynman-Kac formula and Brownian motion. Thanks are also due to the anonymous referee for a substantial improvement in the final presentation. Lastly, the authors would also like to thank Werner Ballmann, and gratefully acknowledge the Max Planck Institute for Mathematics, Bonn for providing ideal working conditions.

\section{Preliminaries: heat equation, Feynman-Kac and Bessel processes}\label{Int}


\subsection{Feynman-Kac formula}\label{F-K}
We begin by stating a Feynman-Kac formula for open connected domains in compact manifolds for the heat equation with Dirichlet boundary conditions. Such formulas seem to be widely known in the community, but since we were unable to find out an explicit reference, we also indicate a line of proof.
\begin{theorem}\label{F-K-T}
	Let $M$ be a compact Riemannian manifold. For any open connected $\Omega \subset M$, $f \in L^2(\Omega)$, we have that 
	\beq
	e^{t\Delta} f(x) = \mathbb{E}_x(f(\omega(t))\phi_\Omega(\omega, t)), t > 0, x \in \Omega,
	\eeq
	where $\omega (t)$ denotes an element of the probability space of Brownian motions starting at $x$, $\mathbb{E}_x$ is the expectation with regards to the measure on that probability space, and 
	$$
	\phi_\Omega(\omega, t) = 
	\begin{cases}
	1, & \text{if } \omega ([0, t]) \subset \Omega\\
	0, & \text{otherwise. }
	\end{cases}
	$$
\end{theorem}

A proof of Theorem \ref{F-K-T} can be constructed in three steps. First, one proves the corresponding statement when $\Omega = M$. This can be found, for example, in 
~\cite{BP}, Theorem 6.2. One can then combine this with a barrier potential method to prove a corresponding statement for domains $\Omega$ with Lipschitz boundary. Lastly, the extension to domains with no regularity requirements on the boundary is achieved by a standard limiting argument. For details on the last two steps, see ~\cite{T}, Chapter 11, Section 3.

\subsection{Euclidean comparability of hitting probabilities}\label{Eu-comp}
Implicit in many of our calculations is the following heuristic: if the metric is perturbed slightly, hitting probabilities of compact sets by Brownian particles are also perturbed slightly, provided one is looking at small distances $r$ and at small time scales $t = O(r^2)$. 

To describe the set up, let $(M, g)$ be a compact Riemannian manifold and cover $M$ by charts $(U_k, \phi_k)$ such that in these charts $g$ is bi-Lipschitz to the Euclidean metric. Consider an open ball $B(p, r) \subset M$, where $r$ is considered small, and in particular, smaller than the injectivity radius of $M$. Let $B(p, r)$ sit inside a chart $(U, \phi)$ and let $\phi (p) = q$ and $\phi (B(p, r)) = B(q, s) \subset \RR^n$. Let $K$ be a compact set inside $B(p, r)$ and let $K^{'} := \phi(K) \subset B(q, s)$. 

Now, let $\psi_K^M(T, p)$ denote the probability that a Brownian motion on $(M, g)$ started at $p$ and killed at a fixed time $T$ hits $K$ within time $T$. $\psi_{K^{'}}^e(t, q)$ is defined similarly for the standard Brownian motion in $\RR^n$ started at $q$ and killed at the same fixed time $T$. Now, we fix the time $T = cr^2$, where $c$ is a constant. The following is the comparability result:
\begin{theorem}\label{compare}
	There exists constants $c_1, c_2$, depending only on $c$ and $M$ such that 
	\beq\label{comp1}
	c_1 \psi_{K^{'}}^e(T, q) \leq \psi_K^M(T, p) \leq c_2 \psi_{K^{'}}^e(T, q).\eeq
\end{theorem}

The proof uses the concept of Martin capacity (see ~\cite{BPP}, Definition 2.1):
\begin{definition}
	Let $\Lambda$ be a set and $\mathcal{B}$ a $\sigma$-field of subsets of $\Lambda$. Given a measurable function $F : \Lambda \times \Lambda \to [0, \infty]$ and a finite measure $\mu$ on $(\Lambda, \Cal{B})$, the $F$-energy of $\mu$ is 
	\[
	I_F(\mu) = \int_{\Lambda}\int_{\Lambda}F(x, y)d\mu(x)d\mu(y).
	\]
	The capacity of $\Lambda$ in the kernel $F$ is 
	\beq
	\text{Cap}_F(\Lambda) = \left[ \inf_\mu I_F(\mu)\right]^{-1},
	\eeq
	where the infimum is over probability measures $\mu$ on $(\Lambda, \Cal{B})$, and by convention, $\infty^{-1} = 0$. 
\end{definition} 

Now we quote the following general result, which is Theorem 2.2 in ~\cite{BPP}.
\begin{theorem}
	Let $\{X_n\}$ be a transient Markov chain on the countable state space $Y$ with initial state $\rho$ and transition probabilities $p(x, y)$. For any subset $\Lambda$ of $Y$, we have
	\beq
	\frac{1}{2}\text{Cap}_M(\Lambda) \leq \mathbb{P}_\rho[\exists n \geq 0 : X_n \in \Lambda] \leq \text{Cap}_M(\Lambda),\eeq
	where $M$ is the Martin kernel $M(x, y) = \frac{G(x, y)}{G(\rho, y)}$, and $G(x, y)$ denotes the Green's function.
\end{theorem} 

For the special case of Brownian motions, this reduces to (see Proposition 1.1 of ~\cite{BPP} and Theorem 8.24 of ~\cite{MP}):
\begin{theorem}
	Let $\{B(t): 0 \leq t \leq T\}$ be a transient Brownian motion in $\RR^n$ starting from the point $\rho$, and $A \subset D$ be closed, where $D$ is a bounded domain. Then,
	\beq
	\frac{1}{2}\text{Cap}_M(A) \leq \mathbb{P}_\rho\{B(t) \in A \text{ for some } 0 < t \leq T\} \leq \text{Cap}_M(A).\eeq
\end{theorem}

An inspection of the proofs reveals that they go through with basically no changes on a compact Riemannian manifold $M$, when the Brownian motion is killed at a fixed time $T = cr^2$, and the Martin kernel $M(x, y)$ is defined as $\frac{G(x, y)}{G(\rho, y)}$, with $G(x, y)$ being the ``cut-off'' Green's function defined as follows: if $h_M(t, x, y)$ is the heat kernel of $M$,
\[
G(x, y) := \int_0^T h_M(t, x, y)dt.
\]

Now, to state it formally, in our setting, we have
\begin{theorem}
	\beq
	\frac{1}{2}\text{Cap}_M(K) \leq \psi_K^M(T, p) \leq \text{Cap}_M(K).
	\eeq
\end{theorem}

Now, let $h_{\RR^n}(t, x, y)$ denote the heat kernel on $\RR^n$. To prove Theorem \ref{compare}, it suffices to show that for $y \in K$, and $y^{'} = \phi(y) \in K^{'}$, we have constants $C_1, C_2$ (depending on $c$ and $M$) such that 
\beq\label{comp}
C_1\int_0^T h_{\RR^n}(t, q, y^{'})dt \leq \int_0^T h_M(t, p, y)dt \leq C_2 \int_0^T h_{\RR^n}(t, q, y^{'})dt.\eeq

In other words, we need to demonstrate comparability of Green's functions ``cut off'' at time $T = cr^2$. Recall that we have the following Gaussian two-sided heat kernel bounds on a compact manifold (see, for example, Theorem 5.3.4 of \cite{Hs} for the lower bound and Theorem 4 of \cite{CLY} for the upper bound, also (4.27) of \cite{GS}): for all $(t, p, y) \in (0, 1) \times M \times M$, and positive constants $c_1, c_2, c_3, c_4$ depending only on the geometry of $M$, 
$$ \frac{c_3}{t^{n/2}}e^{\frac{-c_1d(p,y)^2}{4t}} \leq h_M(t, p, y) \leq \frac{c_4}{t^{n/2}}e^{\frac{-c_2d(p,y)^2}{4t}},$$ 
where $d$ denotes the distance function on $M$. Then, using the comparability of the distance function on $M$ with the Euclidean distance function (which comes via metric comparability in local charts),
for establishing (\ref{comp}), it suffices to observe that for any positive constant $c_5$, we have that 
\[
\int_0^{cr^2}t^{-\frac{n}{2}}e^{-\frac{c_5r^2}{4t}}\;dt = \frac{2^{n - 2}}{c_5^{\frac{n}{2} - 1}}\frac{1}{r^{n-2}}\Gamma\Big(\frac{n}{2}-1,\frac{c_5}{4c}\Big),\] where $\Gamma(s,x)$ is the (upper) incomplete Gamma function. Since $r$ is a small constant chosen independently of $\lambda$, we observe that $C_1, C_2$ are constants in (\ref{comp}) depending only on $c, c_1, c_2, c_3, c_4, c_5, r$ and $M$, which finally proves (\ref{comp1}).

\begin{remark}
	Theorem \ref{compare} is implicit in \cite{St}, but it was not precisely stated or proved there. Since we are unable to find an explicit reference, here we have given a formal statement 
	and indicated a proof. We believe that the statement of Theorem \ref{compare} will also be of independent interest for people interested in stochastic analysis on manifolds.
\end{remark}

\subsection{Brownian motion on a manifold and Euclidean Bessel processes}
Using the probabilistic formulation of the heat equation for the study of nodal geometry, we are largely inspired by the methods in ~\cite{St}. Of course, such ideas have appeared in the literature before; for example, they are implicit in ~\cite{GJ}. Here we extend some ideas of Steinerberger with the help of tools from $n$-dimensional Brownian motion.

Given an open subset $V \subset M$, consider the solution $p_t(x)$ to the following diffusion process:
\begin{align*}
(\pa_t - \Delta)p_t(x) & = 0, \text{    } x \in V\\
p_t(x) & = 1, \text{    } x \in \pa V\\
p_0(x) & = 0, \text{    } x \in V.
\end{align*}

By the Feynman-Kac formula (see Subsection \ref{F-K}), this diffusion process can be understood as the probability that a Brownian motion particle started in $x$ will hit the boundary within time $t$. Now, fix an eigenfunction $\varphi$ (corresponding to the eigenvalue $\lambda$) and a nodal domain $\Omega$, so that $\varphi > 0$ on $\Omega$ without loss of generality. Calling $\Delta$ the Dirichlet Laplacian on $\Omega$ and setting $\Phi (t, x) := e^{t\Delta}\varphi (x)$, we see that $\Phi$ solves 
\begin{align}\label{HF}
(\pa_t - \Delta)\Phi (t, x) & = 0, x \in \Omega\nonumber\\
\Phi (t, x) & = 0, \text{  on   } \{\varphi = 0\}\\
\Phi (0, x) & = \varphi (x), \text{    } x \in \Omega.\nonumber
\end{align}

Using the Feynman-Kac formula given by Theorem \ref{F-K-T}, we have,
\beq
e^{t\Delta}f(x) = \mathbb{E}_x(f(\omega(t))\phi_{\Omega}(\omega, t)), t > 0, 
\eeq
where $\omega (t)$ denotes an element of the probability space of Brownian motions starting at $x$, $\mathbb{E}_x$ is the expectation with regards to the measure on that probability space, and 
$$
\phi_\Omega(\omega, t) = 
\begin{cases}
1, & \text{if } \omega ([0, t]) \subset \Omega\\
0, & \text{otherwise. }
\end{cases}
$$

Now, consider a nodal domain $\Omega$ corresponding to the eigenfunction $\varphi$, and consider the heat flow (\ref{HF}). Let $x_0 \in \Omega$ such that $\varphi (x_0) = \V \varphi\V_{L^\infty(\Omega)}$. We use the following upper bound derived in ~\cite{St}:
\begin{align}\label{stein1}
\Phi (t, x) & = e^{-\lambda t}\varphi (x) = \mathbb{E}_x(\varphi(\omega(t))\phi_{\Omega}(\omega, t))\\
& \leq \V \varphi\V_{L^\infty(\Omega)}\mathbb{E}_x(\phi_{\Omega}(\omega, t)) = \V \varphi\V_{L^\infty(\Omega)}(1 - p_t(x))\nonumber.
\end{align}
Setting $t = \lambda^{-1}$ and $x = x_0$, we see that the probability of the Brownian motion starting at an extremal point $x_0$ leaving $\Omega$ within time $\lambda^{-1}$ is $\leq 1 - e^{-1}$. A rough interpretation is that maximal points $ x $ are situated deeply into the nodal domain. Using the notation introduced in the Introduction, the last derived upper estimate translates to $\psi_{M \setminus \Omega}(\lambda^{-1}, x) \leq 1 - e^{-1}$.


Now, we consider an $m$-dimensional Brownian motion of a particle starting at the origin in $\RR^m$, 
and calculate the probability of the particle hitting a sphere $\{x \in \RR^m : \V x\V \leq r\}$ of radius $r$ within time $t$. By a well known formula first derived in ~\cite{Ke}, we see that such a probability is given as follows:
\beq\label{JKF}
\mathbb{P}(\sup_{0 \leq s \leq t}\V B(s)\V \geq r) = 1 - \frac{1}{2^{\nu - 1}\Gamma (\nu + 1)}\sum^\infty_{k = 1}\frac{j_{\nu, k}^{\nu - 1}}{J_{\nu + 1}(j_{\nu, k})}e^{-\frac{j^2_{\nu, k}\text{  t}}{2r^2}}, \quad \nu > -1,
\eeq
where $\nu = \frac{m - 2}{2}$ is the ``order'' of the Bessel process, $J_\nu$ is the Bessel function of the first kind of order $\nu$, and $0 < j_{\nu, 1} < j_{\nu, 2} < .....$ is the sequence of positive zeros of $J_\nu$.

Choose $x = x_0, t = \lambda^{-1}$, as before, and let $r = C^{1/2}\lambda^{-1/2}$, where $C$ is a constant to be chosen later, independently of $\lambda$. Plugging this in (\ref{JKF}) then reads,

\beq\label{JKF1}
\mathbb{P}(\sup_{0 \leq s \leq \lambda^{-1}}\V B(s)\V \geq C\lambda^{-1/2}) = 1 - \frac{1}{2^{\nu - 1}\Gamma (\nu + 1)}\sum^\infty_{k = 1}\frac{j_{\nu, k}^{\nu - 1}}{J_{\nu + 1}(j_{\nu, k})}e^{-\frac{j^2_{\nu, k}}{2C}}, \quad \nu > -1.
\eeq 

We need to make a few comments about the asymptotic behaviour of $j_{\nu, k}$ here. For notational convenience, we write $\alpha_k \sim \beta_k$, as $k \to \infty$ if we have $\alpha_k/\beta_k \to 1$ as $k \to \infty$. ~\cite{Wa}, pp 506, gives the asymptotic expansion 
\beq\label{asy1}
j_{\nu, k} = (k + \nu/2 + 1/4)\pi + o(1) \text{   as  } k \to \infty,
\eeq
which tells us that $j_{\nu, k} \sim k\pi$. Also, from ~\cite{Wa}, pp 505, we have that 
\beq\label{asy2}
J_{\nu + 1}(j_{\nu, k}) \sim (-1)^{k - 1}\frac{\sqrt{2}}{\pi}\frac{1}{\sqrt{k}}.
\eeq

These asymptotic estimates, in conjunction with (\ref{JKF1}), tell us that keeping $\nu$ bounded, and given a small $\eta > 0$, one can choose the constant $C$ small enough (depending on $\eta$) such that 
\beq\label{fund}
\mathbb{P}(\sup_{0 \leq s \leq \lambda^{-1}}\V B(s)\V \geq C\lambda^{-1/2}) > 1 - \eta.\eeq

This estimate plays a role in Section \ref{Perf}. In this context, see also Proposition 5.1.4 of ~\cite{Hs}. 
\section{Admissibility conditions and intersecting surfaces}\label{Perf}



\begin{proof}[Proof of Theorem \ref{babycase}]
If $\varphi_\lambda$ attains its maximum within $\Omega_\lambda$ at $x_0$, we already know from (\ref{stein1}) that 
\beq\label{eq:Heat-Content-Upper-Bound}
\psi_{M \setminus \Omega_\lambda}(\frac{t_0}{\lambda}, x_0) \leq 1 - e^{-t_0}.
\eeq 

By the admissibility condition on $\Sigma_\lambda$ we know that $x_0$ has a unique metric projection on one and only one $\Sigma_\lambda^{i_{x_0}}$ from the collection $\Sigma_\lambda$.

Now, suppose the result is not true. Choose $R, t_0$ small such that Theorem \ref{compare} applies. Choosing $r_0$ sufficiently smaller than $R$, we can find a $\lambda$ such that $\Omega_\lambda$ is contained in a $\frac{r_0}{\sqrt{\lambda}}$-tubular neighbourhood of $\Sigma_\lambda$, denoted by $N_{r_0\lambda^{-1/2}} (\Sigma_\lambda)$. 
From the remarks after Definition \ref{adm}, it follows that $\Omega_\lambda \subseteq N_{r_0\lambda^{-1/2}} (\Sigma_\lambda^{i_{x_0}})$.

We start a Brownian motion at $x_0$ and, roughly speaking, we see that locally the particle has freedom to wander in $n - k$ ``bad directions'', namely the directions normal to $\Sigma_\lambda^{i_{x_0}}$, before it hits $\pa \Omega_\lambda$. That means, we may consider a $(n - k)$-dimensional Brownian motion $B(t)$ starting at $x_0$; see the following diagram:

\includegraphics{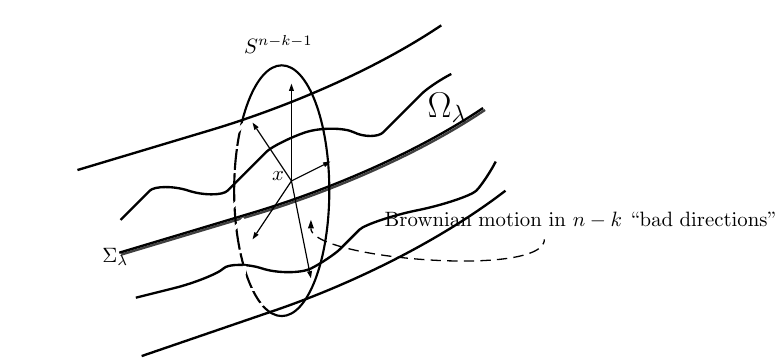}\label{diag1}
\\
\\

More formally, we choose a normal coordinate chart $(U, \phi)$ around $x_0$ adapted to $\Sigma_\lambda^{i_{x_0}}$, where the metric is comparable to the Euclidean metric. We have that $\phi(\Sigma_\lambda^{i_{x_0}}) = \phi(U) \cap \{ \mathbb{R}^k \times \{ 0 \}^{n-k} \} $ and $\phi(N_{r_0\lambda^{-1/2}}(\Sigma_\lambda^{i_{x_0}})) = \phi(U) \cap \{ \mathbb{R}^{k} \times [-\frac{r_0}{\sqrt{\lambda}}, \frac{r_0}{\sqrt{\lambda}}]^{n - k} \}$. 
We take a geodesic ball $B \subset U \subset M$ at $x_0$ of radius $\frac{R}{\sqrt{\lambda}}$. Using the hitting probability notation from Section \ref{Int} and monotonicity with respect to set inclusion we have
\begin{equation}\label{12}
	\psi_{M \setminus \Omega_\lambda}\left(\frac{t_0}{\lambda}, x_0\right) \geq \psi_{B \setminus \Omega_\lambda} \left(\frac{t_0}{\lambda}, x_0\right) \geq \psi_{B \setminus N_{r_0\lambda^{-1/2}}(\Sigma_\lambda^{i_{x_0}})} \left(\frac{t_0}{\lambda}, x_0\right),
\end{equation}
and the comparability lemma implies that, if $c = \frac{t_0}{R^2}$, then there exists a constant $C$, depending on $c$ and $M$, such that 

\begin{equation}
	\psi_{B \setminus N_{r_0\lambda^{-1/2}}(\Sigma_\lambda^{i_{x_0}})}\left(\frac{t_0}{\lambda}, x_0\right) \geq C \psi^e_{\phi(B \setminus N_{r_0\lambda^{-1/2}}(\Sigma_\lambda^{i_{x_0}}))} \left(\frac{t_0}{\lambda}, \phi(x_0)\right),
\end{equation}
where $\psi^e$ denotes the hitting probability in Euclidean space. 
We denote $N_{r_0\lambda^{-1/2}}^e := \phi(N_{r_0\lambda^{-1/2}}(\Sigma_\lambda^{i_{x_0}}))$.

Let us consider the ``solid cylinder'' $S = B_\frac{R}{\sqrt{\lambda}} \times B_\frac{r_0}{\sqrt{\lambda}}$, a product of $k$ and $n-k$ dimensional Euclidean balls centered at $\phi(x_0)$. $S$ is clearly the largest cylinder contained in $ N^e_{r_0\lambda^{-1/2}} \cap B$. We denote $S = B_1 \times B_2$ for convenience. 
By monotonicity,


\beq
 \psi^e_{\phi(B \setminus N_{r_0\lambda^{-1/2}}(\Sigma_\lambda^{i_{x_0}}))} \left(\frac{t_0}{\lambda}, \phi(x_0)\right) \geq  \psi^e_{B_1 \times \pa B_2 }\left(\frac{t_0}{\lambda}, \phi(x_0)\right).
\eeq

If $B(t) = (B_1(t),..., B_n(t))$ is an $n$-dimensional Brownian motion, the components $B_i(t)$'s are independent Brownian motions (see, for example, Chapter 2 of ~\cite{MP}). Denoting by $\Cal{B}_k(t)$ and $\Cal{B}_{n - k}(t)$ the projections of $B(t)$ onto the first $k$ and last $n - k$ components respectively, it follows that
\begin{align*}
	\psi^e_{B_1 \times \pa B_2 }\left(\frac{t_0}{\lambda}, \phi(x_0)\right) & \geq \mathbb{P}(\sup_{0 \leq s \leq t_0\lambda^{-1}} \V \Cal{B}_{k}(t)\V \leq \frac{R}{\sqrt{\lambda}}).\mathbb{P}(\sup_{0 \leq s \leq t_0\lambda^{-1}} \V \Cal{B}_{n-k}(t)\V \geq \frac{r_0}{\sqrt{\lambda}})\\
	& \geq c_k \mathbb{P}(\sup_{0 \leq s \leq t_0\lambda^{-1}} \V \Cal{B}_{n-k}(t)\V \geq \frac{r_0}{\sqrt{\lambda}}),
\end{align*}
where $c_k$ is a constant depending on $k$ and the ratio $t_0/R^2$, and can be calculated explicitly from (\ref{JKF1}).

Using the estimate in Section \ref{Int}, we may take $r_0 \leq R$ sufficiently small so that
\beq\label{16}
	\mathbb{P}(\sup_{0 \leq s \leq t_0\lambda^{-1}} \V \Cal{B}_{n-k}(t)\V \geq 	\frac{r_0}{\sqrt{\lambda}}) > 1 - \varepsilon,
\eeq
where  $\varepsilon$ is sufficiently small. Keeping $c = \frac{t_0}{R^2}$ and (hence) $C$ fixed, we take $t_0$ small enough and $r_0 \leq R$ appropriately, so that (\ref{16}) contradicts (\ref{12}) and the fact that $\psi_{M \setminus \Omega_\lambda}(t_0\lambda^{-1}, x) \leq 1 - e^{-t_0}$.

\end{proof}
\begin{remark}
Note that the constant $r_0$ above is independent of $\Sigma_\lambda$; in other words, the same constant $r_0$ will work for Theorem \ref{babycase} as long as the surface is admissible up to a wavelength distance. Indeed, this results from the fact that $r_0$ depends only on the diffusion process associated to the Brownian motion, and is an inherent property of the manifold itself. 
\end{remark}

Now we address the generalizations of Theorem \ref{babycase} for collections $\Sigma_\lambda$ which are more complicated, namely  
we assume $\Sigma_\lambda$ is a  $\alpha $-admissible collection in the sense of Definition \ref{def:alpha-adm-col}.

\begin{proof}[Proof of Theorem \ref{stein}]
By assumption, we have an $ \alpha $-admissible collection $\Sigma_\lambda := \cup^{m_\lambda}_{i = 1} \Sigma^i_\lambda$.

Let us assume the contrary - if the statement is not true, we may select an arbitrarily small $ r_0 > 0 $ and find a corresponding inscribed nodal domain $\Omega_\lambda \subset N_{r_0 \lambda^{-1/2}}(\Sigma_\lambda)$.

As before, we choose a point $ x_0 \in \Omega_\lambda $ such that
$ \varphi_\lambda (x_0) = \max_{x \in \Omega_\lambda} |\varphi_\lambda|$. Monotonicity of the hitting probability function $\psi_K(.,.)$ with respect to set inclusion in $K$, as well as the $\alpha$-admissibility imply that

\begin{align}\label{t0r0}
	\psi_{M \setminus \Omega_\lambda}(t, x_0)
	& \geq \psi_{B(x_0, 2r_0 \lambda^{-1/2}) \setminus \Omega_\lambda }(t,x_0) \\ \nonumber
	& \geq \psi_{B(x_0, 2r_0 \lambda^{-1/2}) \setminus N_{r_0 \lambda^{-1/2}}(\Sigma_\lambda)}(t, x_0) \\ \nonumber
	& = \psi_{\partial \left( B(x_0, 2r_0 \lambda^{-1/2}) \setminus N_{r_0 \lambda^{-1/2}}(\Sigma_\lambda) \right)}(t, x_0) \\ \nonumber
	& \geq \psi_{\partial B(x_0, 2r_0 \lambda^{-1/2}) \setminus N_{r_0 \lambda^{-1/2}}(\Sigma_\lambda)}(t, x_0) \\ \nonumber
	& \geq \alpha \psi_{\partial B(x_0, 2r_0 \lambda^{-1/2})}(t, x_0),
\end{align}
where we introduce the constant $ \alpha > 0 $ coming from the $ \alpha $-admissibility condition. Moreover, following Definition \ref{def:Alpha-Admissibility} of $ \alpha $-admissibility, in (\ref{t0r0}) we also assume that the radius $ \frac{r_0}{\sqrt{\lambda}} $ is sufficiently small and that  $t := \frac{t_0}{\lambda}$ with $ t_0 := 4r_0^2 $. 

\includegraphics{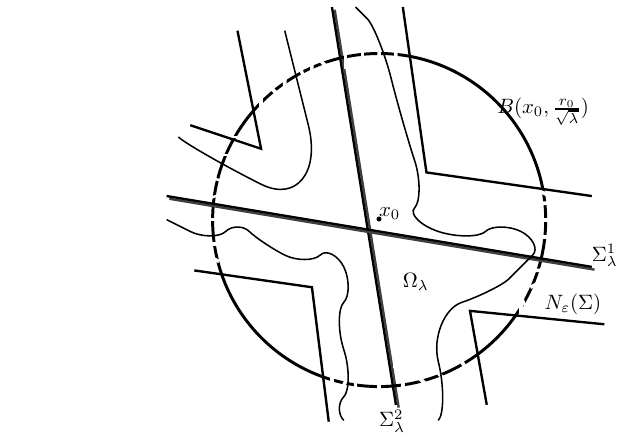}\label{diag2}

The latter estimate $ (\ref{t0r0}) $ implies, in particular, that

\beq\label{t0r01}
	\frac{\psi_{M \setminus \Omega_\lambda}(t, x_0)}{ \psi_{M \setminus B(x_0, 2r_0 \lambda^{-1/2})}(t, x_0)} = \frac{\psi_{M \setminus \Omega_\lambda}(t, x_0)}{\psi_{\partial B(x_0, 2r_0 \lambda^{-1/2})}(t, x_0)} \geq \alpha.
\eeq

We now observe that by setting $t = \frac{t_0}{\lambda}$ we still have the freedom to choose $ t_0 $. We show that we can select $t_0$ such that (\ref{t0r01}) is violated. To this end we observe that the upper bound (\ref{eq:Heat-Content-Upper-Bound}) along with (\ref{JKF1}) and Theorem \ref{compare} give:

\begin{align}\label{wtf}
	\frac{\psi_{M \setminus \Omega_\lambda}(\frac{t_0}{\lambda}, x)}{\psi_{M \setminus B(x_0, 2r_0 \lambda^{-1/2})}(\frac{t_0}{\lambda}, x)} & \lesssim \frac{1 - e^{-t_0}}{1 - \frac{1}{2^{\nu - 1}\Gamma (\nu + 1)}\sum^\infty_{k = 1}\frac{j_{\nu, k}^{\nu - 1}}{J_{\nu + 1}(j_{\nu, k})}e^{-\frac{j^2_{\nu, k}t_0}{2r_0^2}}} 
	\\ & = \frac{1 - e^{-t_0}}{1 - \frac{1}{2^{\nu - 1}\Gamma (\nu + 1)}\sum^\infty_{k = 1}\frac{j_{\nu, k}^{\nu - 1}}{J_{\nu + 1}(j_{\nu, k})}e^{-2j^2_{\nu, k}}} \nonumber\\ & = \frac{1 - e^{-t_0}}{\tilde{C}}\nonumber.
\end{align}

Now, we choose $ t_0 = 4r_0^2 $ small enough, so the last estimate yields a contradiction with (\ref{t0r01}). This proves the theorem. 
\end{proof}

\begin{remark}
	We wish to comment that in the above proof, it is not essential to look at the nodal domain only around the maximum point $x_0$. Given a pre-determined positive constant $\beta$, choose a point $y \in \Omega_\lambda$ such that $\varphi_\lambda (y) \geq \beta \varphi_\lambda (x_0)$. Arguing similarly as in (\ref{stein1}), we see that $\psi_{M \setminus \Omega_\lambda}(t, y) \leq 1 - \beta e^{-t_0}$. Following the computations in (\ref{wtf}), we get a constant $r_0$ (depending on $\beta$) such that $
	\frac{1 - \beta e^{-t_0}}{\tilde{C}} < \alpha$, giving a contradiction. Also, it is clear that in Definitions \ref{adm} and \ref{def:Alpha-Admissibility}, we do not actually need the submanifolds in the family $\Sigma_\lambda$ to be smooth, and the proofs of Theorems \ref{babycase} and \ref{stein} work with submanifolds of much lower regularity (for example, $C^1$ submanifolds).  
\end{remark}

\section{Large ball at a max point}\label{LBMP}

In this section we discuss the asymptotic thickness of nodal domains around extremal points of eigenfunctions. More precisely, let us consider a fixed nodal domain $ \Omega_\lambda $ corresponding to the eigenfunction $ \varphi_\lambda $.
Let $ x_0 \in \Omega_\lambda  $ be such that
\begin{equation}
	\varphi_\lambda (x_0) = \max_{x \in \Omega_\lambda} |\varphi_\lambda|.
\end{equation}

In the case $ \dim M = 2$, it was shown in Section 3 of \cite{Man1} that at such maximal points $ x_0 $ one can
fully inscribe a large ball of wavelength radius (i.e $ \sim \frac{1}{\sqrt{\lambda}} $)  into the nodal domain. 
In other words for Riemannian surfaces, one has that
\begin{equation}
	\frac{C_1}{\sqrt{\lambda}} \leq \text{inrad }(\Omega_\lambda) \leq \frac{C_2}{\sqrt{\lambda}},
\end{equation}
where $ C_i $ are constants depending only on $ M $. Note that the proof for this case, as carried out in \cite{Man1} by following ideas in ~\cite{NPS}, makes use of essentially $2$-dimensional tools (conformal coordinates and quasi-conformality), which are not available in higher dimensions.

To our knowledge, in higher dimensions the sharpest known bounds on the inner radius of a nodal domain appear in ~\cite{Man2} (Theorem 1.5) and state that:
\begin{equation} \label{eq:Mangoubi-Inrad}
	\frac{C_1}{\lambda^{\alpha(n)}} \leq \text{inrad }(\Omega_\lambda) \leq \frac{C_2}{\sqrt{\lambda}},
\end{equation}
where $ \alpha (n) := \frac{1}{4}(n-1)+\frac{1}{2n} $. A question of current investigation is whether the last lower bound on $ \text{inrad }(\Omega_\lambda) $ in higher dimensions is optimal. 

Here we exploit heat equation and Brownian motion techniques to show that at least, one can expect to ``almost'' inscribe a large ball having radius to the order of $\frac{1}{\sqrt{\lambda}}$, in all dimensions. 
	
Now we prove Theorem \ref{th:Large-Inscribed-Ball}:

\begin{proof}

	We denote $ t' := \frac{t_0}{\lambda} $, and thus $ \psi_{M \setminus \Omega_\lambda} (t',x) \leq 1 - e^{-t_0} $, where $t_0$ is a small constant to be chosen suitably later. 
	 
	Now, choosing $ t_0 $ small enough, and using monotonicity, we have,
	
	\begin{equation}\label{21}
		\psi_{B(x_0, r_0 \lambda^{-1/2}) \setminus \Omega_\lambda}(t, x_0) < \psi_{M \setminus \Omega_\lambda}(t, x_0) < \epsilon.
	\end{equation}
	
	For convenience, let us denote $E_{r_0} := B(x_0, r_0 \lambda^{-1/2}) \backslash \Omega_\lambda $ - a relatively compact set. Observe that Theorem \ref{compare} applies to open balls and compact subsets contained in open balls. To adapt to the setting of Theorem \ref{compare}, choose a number $r_0^{'} < r_0$ such that $B(x_0, r'_0 \lambda^{-1/2})$ satisfies $$\frac{\text{Vol}\left(B(x_0, r_0 \lambda^{-1/2}) \setminus B(x_0, r'_0 \lambda^{-1/2})\right)}{\text{Vol}\left(B(x_0, r_0 \lambda^{-1/2})\right)} < \epsilon.$$ 
	
	Call $E_{r_0^{'}} := \overline{E_{r_0} \cap B(x_0, r'_0 \lambda^{-1/2})}$. Observe that proving that $\frac{\text{Vol}(E_{r^{'}_0})}{\text{Vol}(B(x_0, r_0 \lambda^{-1/2}))} < \epsilon$ will imply that $\frac{\text{Vol}(E_{r_0})}{\text{Vol}(B(x_0, r_0 \lambda^{-1/2}))} < 2\epsilon$, which is what we want.
	
	Now, we would like to compare the volumes of the two sets $E_{r^{'}_0}$ and $B(x_0, r_0 \lambda^{-1/2})$.
	Let $r = \frac{r_0}{\sqrt{\lambda}}$. Recall from ~\cite{GS}, Remark 4.1, the following inequality:
\beq\label{Maz}
	c\frac{\text{cap}(E_{r^{'}_0})r^2}{\text{Vol}(B(x_0, r_0 \lambda^{-1/2}))}e^{-C\frac{r^2}{t'}} \leq \psi_{E_{r_0}}(t', x_0) < \epsilon,
\eeq
where $\text{cap}(K)$ denotes the $2$-capacity of the set $K \subset M$, and $0 < t^{'} < 2r^2 $ (see also Equation (3.20) of ~\cite{GS}). Recall that the $2$-capacity of a set $K \subset M$ is defined as 
$$
\text{cap}(K) = \inf_{\eta|_K \equiv 1, \eta \in C^\infty(M)} \int_M |\nabla \eta|^2 dM.$$
Formally, (\ref{Maz}) holds on complete non-compact non-parabolic manifolds, which includes $\RR^n, n \geq 3$. 
For bringing in our comparability result Theorem \ref{compare}, we fix the ratio $\frac{t'}{r^2} = \frac{1}{3}$, say, and then choose $t_0$ small enough that (\ref{21}) still works. Now (\ref{Maz}) applies, albeit with a new constant $c$ as determined by the ratio $t/r^2$ and Theorem \ref{compare}. 


Now, to rewrite the capacity term in (\ref{Maz}) in terms of volume, we bring in the following ``isocapacitary inequality'' (see ~\cite{Maz}, Section 2.2.3):
\beq\label{II}
\text{cap}(E_{r_0}) \geq C'\text{Vol}(E_{r_0})^{\frac{n - 2}{n}}, n \geq 3,
\eeq
where $C'$ is a constant depending only on the dimension $n$. We note that the isocapacitary inequality (in combination with a suitable Poincare inequality) lies at the heart of the currently optimal inradius estimates, as derived by Mangoubi in \cite{Man2}.

Clearly, (\ref{Maz}) and (\ref{II}) together give
\beq \label{eq:Volume-Ratio}
\left( \frac{\text{Vol}(E_{r_0})}{\text{Vol}(B(x_0, r_0 \lambda^{-1/2}))} \right)^{\frac{n-2}{n}} \lesssim \frac{\text{cap}(E_{r_0})r^2}{\text{Vol}(B(x_0, r_0 \lambda^{-1/2}))} \lesssim \psi_{E_{r_0}}(t, x) < \epsilon.
\eeq
The last inequalities contain constants depending only on $ M $, so by taking $ \epsilon $ even smaller we can arrange $\frac{\text{Vol}(E_{r_0})}{\text{Vol}(B(x_0, r_0 \lambda^{-1/2}))} < \epsilon_0$ for any initially given $ \epsilon_0 $.
\end{proof}
\begin{remark}
We note that the heat equation method does not distinguish between a general domain and a nodal domain. This means that we cannot rule out the situation where $B(x_0, \frac{r_0}{\sqrt{\lambda}}) \setminus \Omega_\lambda$ is a collection of ``sharp spikes'' entering into $B(x_0, \frac{r_0}{\sqrt{\lambda}})$. Indeed the probability of a Brownian particle hitting a spike, no matter how ``thin'' it is, or how far from $x_0$ it is, is always non-zero, a fact related to the infinite speed of propagation of heat diffusion. This is consistent with the heuristic discussed in ~\cite{H} and ~\cite{L}.
\end{remark}

Now we establish Corollary \ref{cor:Large-Ball}. First, we recall the following result, which gives a bound on the asymmetry between the volumes of positivity and negativity sets, as developed in \cite{Man2}:
	\begin{theorem}{\cite{Man2}} \label{th:Asymmetry}
		Let $ B $ be a geodesic ball, so that $ \left( \frac{1}{2}B \cap \{ \varphi_\lambda = 0 \} \right) \neq \emptyset$ with $ \frac{1}{2} B $ denoting the concentric ball of half radius. Then
		\begin{equation}\label{29}
			\frac{\Vol(\{\varphi_\lambda > 0 \}\cap B)}{\Vol(B)} \geq \frac{C}{\lambda^\frac{n-1}{2}}.
		\end{equation}
	\end{theorem}

\begin{proof}[Proof of Corollary \ref{cor:Large-Ball}]
	It suffices to combine the estimate (\ref{eq:Volume-Ratio}) with (\ref{29}).
	
	Let $ r:= \frac{r_0}{\sqrt{\lambda}} $ be the radius of the largest inscribed ball in the nodal domain at $ x_0 $. Noting that $\{ \varphi_\lambda < 0\} \subseteq  E_{r_0}$ and combining Theorem \ref{th:Asymmetry} for $ B_{x_0}(2r) $ with (\ref{eq:Volume-Ratio}), we get:
	\begin{equation}
		\left( \frac{C}{\lambda^\frac{n-1}{2}} \right)^{\frac{n-2}{n}} \leq \left( \frac{\text{Vol}(E_{r_0})}{\text{Vol}(B(x_0, r_0 \lambda^{-1/2}))} \right)^{\frac{n-2}{n}} \leq 1 - e^{-\sqrt{1/3} r_0^2}
	\end{equation}
	Expanding the right hand side in Taylor series and rearranging finishes the proof.
\end{proof}

\begin{remark}
	An inspection of the proof of Theorem \ref{th:Large-Inscribed-Ball} reveals that one can take $\epsilon = r_0^{\frac{2n}{n - 2}}$. In other words, the relative volume of the error set $E_{r_0}$ decays as $r_0^{\frac{2n}{n - 2}}$ as $r_0 \to 0$. This is slightly better than the scaling prescribed by Corollary 2 of \cite{L}.
\end{remark}

\begin{remark}
	There is a sizeable literature around optimizing the fundamental frequency of the complement of an obstacle inside a domain (for example, see \cite{HKK} and references therein). As an explicit special case, consider a convex domain $\Omega \subset \RR^n$ and a small ball $B \subseteq \Omega$. The question is to find possible placements of translate $x + B$ inside $\Omega$ such that $\lambda_1(\Omega \setminus (x + B))$ is maximized. For certain applications of Theorem \ref{th:Large-Inscribed-Ball} towards such questions, we refer to \cite{GM1}. 
\end{remark}

\end{document}